\documentclass[12pt,amsfonts]{article}

\input epsf
\usepackage{amsthm}
\usepackage{amssymb}
\usepackage{amsmath}
\usepackage{amsfonts}
\input amssym.def

\begin{document}
\title{Low regularity solutions of two fifth-order KdV type equations}
\author{ Wengu Chen$^1$, Junfeng Li$^{2}$, Changxing Miao$^1$ and Jiahong Wu$^3$
\\$^1$Institute of Applied Physics
and Computational Mathematics\\P.O.Box 8009, Beijing 100088,
China\\$^2$College of Mathematics, Beijing Normal University
\\Beijing 100875, China\\$^3$Department of Mathematics, Oklahoma State
University\\
Stillwater, OK 74078}

\date{E-mail:\,chenwg@iapcm.ac.cn,\, junfli@yahoo.com.cn,\, miao$_-$changxing@iapcm.ac.cn,\,
jiahong@math.okstate.edu}

\maketitle

\newtheorem{thm}{Theorem}[section]
\newtheorem{thmA}[thm]{Theorem A}
\newtheorem{cor}[thm]{Corollary}
\newtheorem{prop}[thm]{Proposition}
\newtheorem{define}[thm]{Definition}
\newtheorem{rem}[thm]{Remark}
\newtheorem{example}[thm]{Example}
\newtheorem{lemma}[thm]{Lemma}
\def\theequation{\thesection.\arabic{equation}}

\vspace{.15in} \noindent {\bf Abstract:} The Kawahara and modified
Kawahara equations are fifth-order KdV type equations and have been
derived to model many physical phenomena such as gravity-capillary
waves and magneto-sound propagation in plasmas. This paper
establishes the local well-posedness of the initial-value problem
for Kawahara equation in $H^s({\mathbf R})$ with  $s>-\frac74$ and
the local well-posedness for the modified Kawahara equation in
$H^s({\mathbf R})$ with $s\ge-\frac14$. To prove these results, we
derive a fundamental estimate on dyadic blocks for the Kawahara
equation through the $[k; Z]$ multiplier norm method of Tao
\cite{Tao2001} and use this to obtain new bilinear and trilinear
estimates in suitable Bourgain spaces.

\vspace{.2in} \noindent {\bf AMS (MOS) Numbers:} 35Q53, 35B30, 76B45

\vspace{.1in} \noindent {\bf Keywords:} low-regularity
well-posedness, Kawahara equation, modified Kawahara equation

\section{Introduction}
\setcounter{equation}{0} \label{sec:1}

This paper is mainly concerned with the local well-posedness of the
initial-value problems (IVP) for the Kawahara equation
\begin{equation}\label{kawa}
\left\{
\begin{array}{l}
u_t + uu_x + \alpha u_{xxx} + \beta u_{xxxxx}=0, \quad x,t\in
{\mathbf
R},\\
u(x,0) =u_0(x).
\end{array}
\right.
\end{equation}
and for the modified Kawahara equation
\begin{equation}\label{mkawa}
\left\{
\begin{array}{l}
u_t + u^2u_x + \alpha u_{xxx} + \beta u_{xxxxx}=0, \quad x,t\in
{\mathbf
R},\\
u(x,0) =u_0(x),
\end{array}
\right.
\end{equation}
where $\alpha$ and $\beta$ are real constants and $\beta\neq 0$.
Attention will be focused on solutions in Sobolev spaces of negative
indices. These fifth-order KdV type equations arise in modeling
gravity-capillary waves on a shallow layer and magneto-sound
propagation in plasmas (see e.g. \cite{Ka}, \cite{KO}).

\vspace{.1in} The well-posedness issue on these fifth-order KdV type
equations has previously been studied by several authors. In
\cite{Po}, Ponce considered a general fifth-order KdV equation
$$
u_t + u_x + c_1 uu_x + c_2u_{xxx} + c_3u_xu_{xx} + c_4uu_{xxx}+ c_5
u_{xxxxx} =0, \quad x,\,t\in {\mathbf R}
$$
and established the global well-posedness of the corresponding IVP
for any initial data in $H^4({\mathbf R})$. In \cite{KPV94a} and
\cite{KPV94b}, Kenig, Ponce and Vega studied the local
well-posedness of the IVP for the following odd-order equation
$$
u_t+ \partial_x^{2j+1}u + P(u,\partial_x u, \cdots,
\partial_x^{2j} u) =0
$$
where $P$ is a polynomial having no constant or linear terms. They
obtained the local well-posedness for
$$
u_0\in H^s({\mathbf R})\cap L^2(|x|^m dx),
$$
where $s,m\in {\mathbf Z}^+$. Cui, Deng and Tao in \cite{CDT}
established the local well-posedness in $H^s$ with $ s>-1$ for the
Kawahara equation. Wang, Cui and Deng in a very recent work
\cite{WCD} obtained the local well-posedness in $H^s$ with $ s\ge
-\frac 75$ for the Kawahara equation by the same method as in
\cite{CDT}. Their method is derived from that of Kenig, Ponce and
Vega \cite{KPV96} for the cubic KdV equation. In \cite{Taocui}, Tao
and Cui studied the low regularity solutions of the modified
Kawahara equation (\ref{mkawa}) and proved the local well-posedness
of the IVP in any Sobolev space $H^s({\mathbf R})$ with $s\ge \frac
14$ by employing an approach of Kenig-Ponce-Vega for the generalized
KdV equations \cite{KPV93a}.

\vspace{.1in} Our goal here is to improve the existing low
regularity well-posedness results. To this end, we first derive a
fundamental estimate on dyadic blocks (see Lemma 3.2 below) for the
Kawahara equation by following the idea in the $[k;\,Z]$-multiplier
norm method introduced by Tao \cite{Tao2001}. We then apply this
fundamental estimate to establish new bilinear and trilinear
estimates in Bourgain spaces. Combining these estimates with a
contraction mapping argument, we are able to prove the following two
theorems.

\begin{thm}\label{major}
Let $s> -\frac74$ and $u_0\in H^s({\mathbf R})$. Then there exist
$b=b(s)\in(\frac 12,\,1)$ and $T=T(\|u_0\|_{H^s})>0$ such that the
IVP(1.1) has a unique solution on $[0,T]$ satisfying
$$
u\in C([0,T]; H^s({\mathbf R}))\quad\mbox{and}\quad u\in X_{s,b},
$$
where $X_{s,b}$ is a Bourgain type space {\rm(}defined in the next
section{\rm)}. In addition, the dependence of $u$ on $u_0$ is
Lipschitz.
\end{thm}

\begin{thm}\label{major2}
Let $s\ge -\frac 14$ and $u_0\in H^s({\mathbf R})$. Then there exist
$b=b(s)\in(\frac 12,\,1)$ and $T=T(\|u_0\|_{H^s})>0$ such that the
IVP for the modified Kawahara equation (\ref{mkawa}) has a unique
solution on $[0,T]$ satisfying
$$
u\in C([0,T]; H^s({\mathbf R}))\quad\mbox{and}\quad u\in X_{s,b},
$$
and the dependence of $u$ on $u_0$ is Lipschitz.
\end{thm}
\noindent

The proofs of Theorems \ref{major} and \ref{major2} will be provided
in the subsequent sections.

\section{Linear and bilinear estimates for the Kawahara equation}
\setcounter{equation}{0} \label{sec:2}

This section provides the linear and bilinear estimates for the
Kawahara equation. We start with a few notation. Denote by $W(t)$
the unitary group generating the solution of the IVP for the linear
equation
$$
\left\{
\begin{array}{l}
v_t + \alpha v_{xxx} + \beta v_{xxxxx} =0, \quad x\in {\mathbf
R},\,\,t\in {\mathbf R},\\
v(x,0) =v_0(x).
\end{array}
\right.
$$
That is,
$$
 v(x,t) =W(t) v_0(x) =S_t \ast v_0(x),
$$
where $ \widehat{S}_t = e^{it p(\xi)}$ with $ p(\xi) = -\beta \xi^5
+\alpha \xi^3$, or
$$
S_t(x)  = \int e^{i(x\xi + t p(\xi))}\,d\xi.
$$
For $s,b\in {\mathbf R}$,  let $X_{s,\,b}$ denote the completion of
the functions in $ C_0^\infty$ with respect to the norm
$$
\|f\|^2_{X_{s,b}} =\int \langle \xi\rangle^{2s} \langle \tau -
p(\xi)\rangle^{2b}\, |\widehat{f}(\xi,\tau)|^2 \,d\xi\,d\tau,
$$
where $\langle \xi\rangle =1+|\xi|$. It is easy to verify that
$$
\|f\|_{X_{s,\,b}} = \|J^s \Lambda^b W(-t) f\|_{L^2_{x,\,t}}
$$
where
$$
\widehat{J g}(\xi) = (1+|\xi|) \widehat{g}(\xi),\quad
\widehat{\Lambda h}(\tau) = (1+|\tau|) \widehat{h}(\tau).
$$

Let $\psi\in C_0^\infty$ be a standard bump function and consider
the following integral equation
$$
u(t)= \psi(\delta^{-1}t) W(t) u_0 - \psi(\delta^{-1}t)\int_0^t
W(t-t') u(t') \partial_x u(t')\,dt'.
$$
Denote the right-hand side by ${\cal T}(u)$. The goal is to show
that ${\cal T}(u)$ is contraction on the following complete metric
space $Y$, where
$$
Y= \{ u\in X_{s,\,b}:\,\, \|u\|_{X_{s,\,b}} \le 2c_0
\delta^{(1-2b)/2} \|u_0\|_{H^s}\}
$$
with  metric
$$d(u, v)=\|u-v\|_{X_{s,\,b}},\qquad u, v\in Y,$$
where $c_0$ is the constant appeared in Proposition 2.1. For this
purpose, we need two linear estimates and one bilinear estimate
stated in the following propositions.
\begin{prop}
For $s\in  {\mathbf R}$ and $b>\frac12$,
\begin{eqnarray}
&& \|\psi(\delta^{-1} t) W(t) u_0\|_{X_{s,\,b}} \le c_0
\delta^{(1-2b)/2}\|u_0\|_{H^s}, \nonumber \\
&& \left\|\psi(\delta^{-1} t)\int_0^t W(t-t') f(t')\,dt'
\right\|_{X_{s,\,b}} \le c_0 \delta^{(1-2b)/2}
\|f\|_{X_{s,\,b-1}}.\nonumber
\end{eqnarray}
\end{prop}
The proof of these estimates follows directly from Kenig, Ponce and
Vega \cite{KPV93b}.

\begin{prop}\label{nonl}
For any $s> -\frac74$, there is $b$ satisfying $\frac12<b<1$ such
that
\begin{equation}\label{bil}
\|\partial_x(uv)\|_{X_{s,\,b-1}} \le c_1
\|u\|_{X_{s,\,b}}\|v\|_{X_{s,\,b}},
\end{equation}
\end{prop}
\noindent where $c_1$ is a constant depending on $s$ and $b$ only.

\vspace{.1in} Proposition \ref{nonl} will be proved in Section 4 and
in the next section we introduce Tao's $[k;\,Z]$-multiplier norm
method and prove a fundamental estimate on dyadic blocks for the
Kawahara equation from which a variety of bilinear estimates can be
derived. Once the estimates in Propositions 2.1 and 2.2 are
available, a standard argument then yields that ${\cal T}(u)$ is a
contraction on $Y$.

\section{Fundamental estimate on dyadic blocks for the Kawahara equation}
\setcounter{equation}{0} \label{sec:3}

In this section we introduce Tao's $[k;\,Z]-$multiplier norm method
and establish the fundamental estimate on dyadic blocks, i.e., Lemma
3.2 for the Kawahara equation from which Proposition \ref{nonl} and
other bilinear estimates (see Lemma 5.2 below) could be derived.

Let $Z$ be any abelian additive group with an invariant measure
$d\xi$. For any integer $k\geq 2$, we let $\Gamma_k(Z)$ denote the
hyperplane
$$\Gamma_k(Z):=\{(\xi_1,\cdots,\,\xi_k)\in
Z^k:\,\xi_1+\cdots+\xi_k=0\}$$ which is endowed with the measure
$$\int_{\Gamma_k(Z)}f:=\int_{Z^{k-1}}f(\xi_1,\cdots,\,\xi_{k-1},\,-\xi_1
-\cdots-\xi_{k-1})d\xi_1\cdots d\xi_{k-1}.$$

A $[k;\,Z]-$multiplier is defined to be any function
$m:\,\Gamma_k(Z)\rightarrow {\Bbb C}$ which was introduced by Tao in
\cite{Tao2001}. And the multiplier norm $\|m\|_{[k;\,Z]}$ is defined
to be the best constant such that the inequality
\begin{eqnarray}
\Big |\int_{\Gamma_k(Z)}m(\xi)\prod_{j=1}^kf_j(\xi_j)\Big |\leq
c\prod_{j=1}^k\|f_j\|_{L^2(Z)},
\end{eqnarray}
holds for all test functions $f_j$ on $Z$. Tao systematically
studied this kind of weighted convolution estimates on $L^2$ in
\cite{Tao2001}. To establish the fundamental estimate on dyadic
blocks for the Kawahara equation, we use some notations.

We use $A\lesssim B$ to denote the statement that $A\leq CB$ for
some large constant $C$ which may vary from line to line and depend
on various parameters, and similarly use $A\ll B$ to denote the
statement $A\leq C^{-1}B$. We use $A\sim B$ to denote the statement
that $A\lesssim B\lesssim A$.

Any summations over capitalized variables such as $N_j,\,L_j,\,H$
are presumed to be dyadic, i.e., these variables range over numbers
of the form $2^k$ for $k\in {\Bbb Z}$. Let $N_1,\,N_2,\,N_3>0$. It
will be convenient to define the quantities $N_{max}\geq N_{med}\geq
N_{min}$ to be the maximum, median, and minimum of $N_1,\,N_2,\,N_3$
respectively. Similarly define $L_{max}\geq L_{med}\geq L_{min}$
whenever $L_1,\,L_2,\,L_3>0$. And we also adopt the following
summation conventions. Any summation of the form $L_{max}\sim\cdots$
is a sum over the three dyadic variables $L_1,\,L_2,\,L_3\gtrsim 1$,
thus for instance
$$\sum_{L_{max}\sim H}:=\sum_{L_1,\,L_2,\,L_3\gtrsim 1:\,L_{max}\sim
H}.$$ Similarly, any summation of the form $N_{max}\sim\cdots$ sum
over the three dyadic variables $N_1,\,N_2,\,N_3>0$, thus for
instance
$$\sum_{N_{max}\sim N_{med}\sim
N}:=\sum_{N_1,\,N_2,\,N_3>0:\,N_{max}\sim N_{med}\sim N}.$$ If
$\tau,\,\xi$ and $p(\xi)$ are given, we define
$$\lambda:=\tau-p(\xi).$$ Similarly,
$$\lambda_j:=\tau_j-p(\xi_j),\,j=1,\,2,\,3.$$

In this paper, we do not go further on the general framework of
Tao's weighted convolution estimates. We focus our attention on the
$[3;\,Z]-$multiplier norm estimate for the Kawahara equation. During
the estimate we need the resonance function
\begin{eqnarray}\label{resonance}
h(\xi)=p(\xi_1)+p(\xi_2)+p(\xi_3)=-\lambda_1-\lambda_2-\lambda_3,
\end{eqnarray}
which measures to what extent the spatial frequencies
$\xi_1,\,\xi_2,\,\xi_3$ can resonate with each other.

By dyadic decomposition of the variables $\xi_j,\,\lambda_j$, as
well as the function $h(\xi)$, one is led to consider
\begin{eqnarray}\label{block}
\|X_{N_1,\,N_2,\,N_3;\,H;\,L_1,\,L_2,\,L_3}\|_{[3,\,{\Bbb
R}\times{\Bbb R}]},
\end{eqnarray}
where $X_{N_1,\,N_2,\,N_3;\,H;\,L_1,\,L_2,\,L_3}$ is the multiplier
\begin{eqnarray}
X_{N_1,\,N_2,\,N_3;\,H;\,L_1,\,L_2,\,L_3}(\xi,\,\tau):=\chi_{|h(\xi)|\sim
H}\prod_{j=1}^3\chi_{|\xi_j|\sim N_j}\chi_{|\lambda_j|\sim L_j}.
\end{eqnarray}

From the identities
$$\xi_1+\xi_2+\xi_3=0$$
and
$$\lambda_1+\lambda_2+\lambda_3+h(\xi)=0$$
on the support of the multiplier, we see that
$X_{N_1,\,N_2,\,N_3;\,H;\,L_1,\,L_2,\,L_3}$ vanishes unless
\begin{eqnarray}
N_{max}\sim N_{med},
\end{eqnarray}
and
\begin{eqnarray}
L_{max}\sim \max(H,\,L_{med}).
\end{eqnarray}

 From the definition of the resonance function, i.e.,
(\ref{resonance}), we obtain the following algebraic smoothing
relation
\begin{lemma} If $N_{max}\sim N_{med}\gtrsim 1$, then
\begin{equation}\label{3sig}
\max\{|\lambda_1|, |\lambda_2|,|\lambda_3|\} \gtrsim
N_{\max}^4N_{\min}.
\end{equation}
\end{lemma}
\begin{proof} Noticing that $p(\xi_j) =-\beta \xi_j^5
+\alpha \xi_j^3,\,j=1,\,2,\,3$, we have
\begin{eqnarray*}
h(\xi)&=& -\lambda_1- \lambda_2- \lambda_3
=p(\xi_1)+p(\xi_2)+p(\xi_3)\\
&=& \xi_1\xi_2\xi_3\Big(3\alpha-5\beta
(\xi_1^2+\xi_1\xi_2+\xi_2^2)\Big).
\end{eqnarray*}
Since $\xi_1^2+\xi_1\xi_2+\xi_2^2 \sim \max\{\xi_1^2,\,\xi_2^2\}$,
and if $N_{max}\sim N_{med}\gtrsim 1$ and $\beta\neq 0$, we obtain
that
\begin{eqnarray*}
\max \{|\lambda_1|, |\lambda_2|, |\lambda_3|\} &\ge& \frac{1}{3}
(|\lambda_1 + \lambda_2+\lambda_3|) \gtrsim N_{\max}^4N_{\min}.
\end{eqnarray*}
\end{proof}
Under the condition of Lemma 3.1, we see that we may assume that
\begin{eqnarray}
H\sim N_{\max}^4N_{\min},
\end{eqnarray}
since the multiplier in (3.4
) vanishes otherwise.

Now we are in the position to state the fundamental estimate on
dyadic blocks for the Kawahara equation.
\begin{lemma}\label{fundamental estimate on dyadic blocks} Let
$H,\,N_1,\,N_2,\,N_3,\,L_1,\,L_2,\,L_3>0$ obey (3.5), (3.6), (3.8).

$\diamond$((++)Coherence) If $N_{max}\sim N_{min}$ and $L_{max}\sim
H$, then we have
\begin{eqnarray}\label{estimate1}
(\ref{block})\lesssim L_{min}^{1/2}N_{max}^{-2}L_{med}^{1/2}.
\end{eqnarray}

$\diamond$((+-)Coherence) If $N_2\sim N_3\gg N_1$ and $H\sim
L_1\gtrsim L_2,\,L_3$, then
\begin{eqnarray}\label{estimate2}
(\ref{block})\lesssim
L_{min}^{1/2}N_{max}^{-2}\min(H,\,\frac{N_{max}}{N_{min}}L_{med})^{1/2}.
\end{eqnarray}
Similarly for permutations.

$\diamond$ In all other cases, we have
\begin{eqnarray}\label{estimate3}
(\ref{block})\lesssim
L_{min}^{1/2}N_{max}^{-2}\min(H,\,L_{med})^{1/2}.
\end{eqnarray}
\end{lemma}
\begin{proof} The fundamental estimate on dyadic blocks for the Kawahara equation is new. We
prove it by using the tools Tao developed in \cite{Tao2001}.

In the high modulation case $L_{max}\sim L_{med}\gg H$ we have by an
elementary estimate employed by Tao (see (37) p.861 in
\cite{Tao2001})
$$(\ref{block})\lesssim L_{min}^{1/2}N_{min}^{1/2}\lesssim L_{min}^{1/2}N_{max}^{-2}
N_{min}^{1/2}N_{max}^2\lesssim L_{min}^{1/2}N_{max}^{-2}H^{1/2}.$$
For the low modulation case: $L_{max}\sim H$, by symmetry we may
assume that $L_1\ge L_2\ge L_3$.

By Corollary 4.2 in Tao's paper \cite{Tao2001}, we have
\begin{eqnarray}
(\ref{block})&\lesssim &L_3^{1/2}\Big |\{\xi_2:\,|\xi_2-\xi_2^0|\ll
N_{min};\,|\xi-\xi_2-\xi_3^0|\ll
N_{min};\nonumber\\&&\quad\quad\quad\quad
p(\xi_2)+p(\xi-\xi_2)=\tau+O(L_2)\}\Big |^{1/2}\label{RHD}
\end{eqnarray}
for some $\tau\in{\mathbf R},\,\xi,\,\xi_1^0,\,\xi_2^0,\,\xi_3^0$
satisfying
$$|\xi_j^0|\sim N_j\,(j=1,\,2,\,3);\,|\xi_1^0+\xi_2^0+\xi_3^0|\ll
N_{min};\,|\xi+\xi_1^0|\ll N_{min}.$$ To estimate the right-hand
side of the expression (\ref {RHD}) we shall use the identity
\begin{eqnarray}\label{id1}
p(\xi_2)+p(\xi-\xi_2)=p(\xi)+q(\xi,\,\xi_2)\end{eqnarray}
 where
$$q(\xi,\,\eta)=5\beta\xi\eta(\xi-\eta)(\xi^2-\xi\eta+\eta^2)-3\alpha\xi\eta(\xi-\eta).$$

We need to consider three cases:\,$N_1\sim N_2\sim N_3,\,N_1\sim
N_2\gg N_3$ and $N_2\sim N_3\gg N_1$. The case $N_1\sim N_3\gg N_2$
follows by symmetry. By (\ref{id1}) and (\ref{RHD}), we have
\begin{eqnarray}\label{id2}
p(\xi)+(\xi-\xi_2)\Big(5\beta\xi\xi_2(\xi^2-\xi\xi_2+\xi_2^2)-3\alpha\xi\xi_2\Big)=\tau+O(L_2).
\end{eqnarray}
(i) If $N_1\sim N_2\sim N_3$, we see from (\ref{id2}) that $\xi_2$
variable is contained in one interval of length
$O(L_2N_{max}^{-4})$, and then
$$(\ref{block})\lesssim
L_3^{1/2}L_2^{1/2}N_{max}^{-2}=L_{min}^{1/2}L_{med}^{1/2}N_{max}^{-2},$$
so (3.9) follows.\\ (ii) If $N_1\sim N_2\gg N_3$, the same
computation as in the case (i) gives that
$$(\ref{block})\lesssim
L_{min}^{1/2}L_{med}^{1/2}N_{max}^{-2}.$$ (iii) If $N_2\sim N_3\gg
N_1$, we see from (\ref{id2}) that $\xi_2$ variable is contained in
one interval of length $O(L_2N_{max}^{-3}N_{min}^{-1})$, and then
$$(\ref{block})\lesssim
L_3^{1/2}L_2^{1/2}N_{min}^{-1/2}N_{max}^{-3/2}=
L_{min}^{1/2}L_{med}^{1/2}N_{min}^{-1/2}N_{max}^{-3/2}.$$ But
$\xi_2$ is also contained in an interval of length $\ll N_{min}$.
The claim (3.10) follows.

\end{proof}

\section{Proof of Proposition \ref{nonl}}
\setcounter{equation}{0} \label{sec:4}

This section is devoted to the proof of Proposition \ref{nonl} with
the fundamental estimate on dyadic blocks in Lemma \ref{fundamental
estimate on dyadic blocks}.

\begin{proof} By Plancherel it suffices to show that
\begin{eqnarray}\label{norm1}
\left\|\frac{(\xi_1+\xi_2)<\xi_1>^{-s}<\xi_2>^{-s}<\xi_3>^{s}}{<\tau_1-p(\xi_1)>^{b}
<\tau_2-p(\xi_2)>^{b}<\tau_3-p(\xi_3)>^{1-b}}\right\|_{[3,\,{\Bbb
R}\times{\Bbb R}]}\lesssim 1.
\end{eqnarray}
By dyadic decomposition of the variables $\xi_j,\,\lambda_j\,
(j=1,2,3),\,h(\xi)$, we may assume that $|\xi_j|\sim
N_j,\,|\lambda_j|\sim L_j\,(j=1,2,3),\,|h(\xi)|\sim H$. By the
translation invariance of the $[k;\,Z]$-multiplier norm, we can
always restrict our estimate on $L_j\gtrsim 1\,(j=1,\,2,\,3)$ and
$\max(N_1,N_2,N_3)\gtrsim 1$. The comparison principle and
orthogonality (see Schur's test in \cite{Tao2001},\,p851) reduce the
multiplier norm estimate (\ref{norm1}) to showing that
\begin{eqnarray}\label{id3}
\sum_{N_{max}\sim N_{med}\sim N}\sum_{L_1,\,L_2,\,L_3\gtrsim
1}&&\frac{N_3<N_3>^{s}}{<N_1>^s<N_2>^sL_1^bL_2^bL_3^{1-b}}\nonumber\\
&&\left\|X_{N_1,\,N_2,\,N_3;\,L_{max};\,L_1,\,L_2,\,L_3}\right\|_{[3;\,{\Bbb
R}\times{\Bbb R}]}\lesssim 1
\end{eqnarray}
and
\begin{eqnarray}\label{id4}
\sum_{N_{max}\sim N_{med}\sim N}\sum_{L_{max}\sim L_{med}}\sum_{H\ll
L_{max}}&&\frac{N_3<N_3>^{s}}{<N_1>^s<N_2>^sL_1^bL_2^bL_3^{1-b}}
\nonumber\\
&&\left\|X_{N_1,\,N_2,\,N_3;\,H;\,L_1,\,L_2,\,L_3}\right\|_{[3;\,{\Bbb
R}\times{\Bbb R}]}\lesssim 1
\end{eqnarray}
for all $N\gtrsim 1$. Estimates (\ref{id3}) and (\ref{id4}) will be
accomplished by the fundamental estimate Lemma \ref{fundamental
estimate on dyadic blocks} and some delicate summation.

Fix $N\gtrsim 1$. This implies (3.8). We first prove (\ref{id4}). By
(\ref{estimate3}) we reduce to
\begin{eqnarray}\label{id5}
&&\sum_{N_{max}\sim N_{med}\sim N}\sum_{L_{max}\sim L_{med}\gtrsim
N^4N_{min}} \frac{N_3<N_3>^{s}}{<N_1>^s<N_2>^sL_1^bL_2^bL_3^{1-b}}
L_{min}^{1/2}N_{min}^{1/2}\lesssim 1.
\end{eqnarray}
By symmetry we only need to consider two cases: $N_1\sim N_2\sim
N,\,N_3=N_{min}$ and $N_1\sim N_3\sim N,\,N_2=N_{min}$.

(i) In the first case $N_1\sim N_2\sim N,\,N_3=N_{min}$, the
estimate (\ref{id5}) can be further reduced to
\begin{eqnarray*}
&&\sum_{N_{max}\sim N_{med}\sim N}\sum_{L_{max}\sim L_{med}\gtrsim
N^4N_{min}}\frac{N^{-2s}N_{min}<N_{min}>^s}{L_{min}^bL_{med}^bL_{max}^{1-b}}
L_{min}^{1/2}N_{min}^{1/2}\lesssim 1,
\end{eqnarray*}
then performing the $L$ summations, we reduce to
\begin{eqnarray*}
\sum_{N_{max}\sim N_{med}\sim
N}\frac{N^{-2s}N_{min}^{3/2}<N_{min}>^s}{N^{4}N_{min}}\lesssim 1,
\end{eqnarray*}
which is true if $4+2s>0$. So, (\ref{id5}) is true if $s>-2$.

(ii) In the second case $N_1\sim N_3\sim N,\,N_2=N_{min}$, the
estimate (\ref{id5}) can be reduced to
\begin{eqnarray*}
&&\sum_{N_{max}\sim N_{med}\sim N}\sum_{L_{max}\sim L_{med}\gtrsim
N^4N_{min}}\frac{N}{<N_{min}>^sL_{min}^bL_{med}^bL_{max}^{1-b}}
L_{min}^{1/2}N_{min}^{1/2}\lesssim 1.
\end{eqnarray*}
Before performing the $L$ summations, we need pay a little more
attention to the summation of $N_{min}$. So we reduce to
\begin{eqnarray*}
&&\sum_{N_{max}\sim N_{med}\sim N,\,N_{min}\leq 1}\sum_{L_{max}\sim
L_{med}\gtrsim
N^4N_{min}}\frac{NN_{min}^{1/2}}{L_{min}^{b-1/2}L_{max}^{1/2}}+\\
&&\sum_{N_{max}\sim N_{med}\sim N,\,1\leq N_{min}\leq
N}\sum_{L_{max}\sim L_{med}\gtrsim
N^4N_{min}}\frac{NN_{min}^{1/2-s}}{L_{min}^{b-1/2}L_{max}}\lesssim
1,
\end{eqnarray*}
which is obviously true if $s> -\frac 72$. So, (\ref{id5}) is true
if $s>-\frac 72$.

Now we show the low modulation case (\ref{id3}). In this case
$L_{max}\sim N_{max}^4N_{min}$. We first deal with the contribution
where (\ref{estimate1}) holds. In this case we have
$N_1,\,N_2,\,N_3\sim N\gtrsim 1$, so we reduce to
\begin{eqnarray}\label{id6}
\sum_{L_{max}\sim
N^5}\frac{N^{-s}N}{L_{min}^bL_{med}^bL_{max}^{1-b}}
L_{min}^{1/2}N^{-2}L_{med}^{1/2}\lesssim 1.
\end{eqnarray}
Performing the $L$ summations, we reduce to
$$\frac{1}{N^{1+s}N^{5(1-b)}}\lesssim 1,$$
which is true if $1+s+5(1-b)>0$. So, (\ref{id6}) is true if
$s>-\frac 72$ and $\frac 12<b<\frac{6+s}5$.

Now we deal with the cases where (\ref{estimate2}) holds. By
symmetry we only need to consider two cases
$$\begin{array}{ll}
N\sim N_1\sim N_2\gg N_3;\,& H\sim L_3\gtrsim L_1,\,L_2\\
N\sim N_2\sim N_3\gg N_1;\,& H\sim L_1\gtrsim L_2,\,L_3
\end{array}.$$

In the first case we reduce by (\ref{estimate2}) to
\begin{eqnarray}\label{id7}
&&\sum_{N_3\ll N}\sum_{1\lesssim L_1,\,L_2\lesssim
N^4N_3}\frac{N_3<N_3>^s}{N^{2s}L_1^bL_2^bL_3^{1-b}}
L_{min}^{1/2}N^{-2}\min\Big(N^4N_3,\,\frac{N}{N_3}L_{med}\Big)^{1/2}\lesssim
1.
\end{eqnarray}
Decompose the left-hand side of (\ref{id7}) into the following two
terms:
\begin{eqnarray*}\label{id8}
&&\quad\quad\sum_{N_3\leq 1}\sum_{1\lesssim L_1,\,L_2\lesssim
N^4N_3}\frac{N_3<N_3>^s}{N^{2s}L_1^bL_2^bL_3^{1-b}}
L_{min}^{1/2}N^{-2}\min\Big(N^4N_3,\,\frac{N}{N_3}L_{med}\Big)^{1/2}\\
&&+\sum_{1<N_3\ll N}\sum_{1\lesssim L_1,\,L_2\lesssim
N^4N_3}\frac{N_3<N_3>^s}{N^{2s}L_1^bL_2^bL_3^{1-b}}
L_{min}^{1/2}N^{-2}\min\Big(N^4N_3,\,\frac{N}{N_3}L_{med}\Big)^{1/2}\\
&&=:I_1+I_2.
\end{eqnarray*}
We estimate the above two terms separately.

We first consider the estimate of $I_1$. If $N^4N_3=\frac
N{N_3}L_{med}$, then $N_3=\Big(\frac{L_{med}}{N^3}\Big)^{1/2}$. We
divide two cases: $\Big(\frac{L_{med}}{N^3}\Big)^{1/2}\geq 1$ and
$\Big(\frac{L_{med}}{N^3}\Big)^{1/2}<1$ to estimate $I_1$. When
$\Big(\frac{L_{med}}{N^3}\Big)^{1/2}\geq 1$,
\begin{eqnarray}\label{id9}
I_1&\leq &\sum_{N_3\leq 1}\sum_{1\lesssim L_1,\,L_2\lesssim
N^4N_3}\frac{N_3}{N^{2s}L_{min}^{b-1/2}L_{med}^b(N^4N_3)^{1-b}}
N^{-2}N^2N_3^{1/2}.
\end{eqnarray}
Performing the $N_3$ summation in (\ref{id9}), we have
\begin{eqnarray*}
I_1&\lesssim & \sum_{\substack{1\lesssim L_1,\,L_2\lesssim
N^4\\L_{med}\geq
N^3}}\frac{1}{N^{2s}L_{min}^{b-1/2}L_{med}^{b}N^{4(1-b)}}\lesssim 1
\end{eqnarray*}
if $2s+4-b>0$. That means that $I_1\lesssim 1$ if $s>-\frac 74$ and
$\frac 12<b< 4+2s$.

When $\Big(\frac{L_{med}}{N^3}\Big)^{1/2}<1$,
\begin{eqnarray}\label{id10}
&&I_1\leq \sum_{N_3<(\frac{L_{med}}{N^3})^{1/2}}\sum_{1\lesssim
L_1,\,L_2\lesssim
N^4N_3}\frac{N_3}{N^{2s}L_{min}^{b-1/2}L_{med}^b(N^4N_3)^{1-b}}
N^{-2}N^2N_3^{1/2}\nonumber\\
&& +\quad\sum_{(\frac{L_{med}}{N^3})^{1/2}<N_3<1}\sum_{1\lesssim
L_1,\,L_2\lesssim
N^4N_3}\frac{N_3}{N^{2s}L_{min}^{b-1/2}L_{med}^b(N^4N_3)^{1-b}}
N^{-2}N^{1/2}L_{med}^{1/2}N_3^{-1/2}.\nonumber
\end{eqnarray}
Performing the $N_3$ summation, we have
\begin{eqnarray*}
I_1&\lesssim & \sum_{1\lesssim L_1,\,L_2\lesssim
N^4}\frac{\Big(\frac{L_{med}}{N^3}\Big)^{\frac 12(\frac
12+b)}}{N^{2s}L_{min}^{b-1/2}L_{med}^{b}N^{4(1-b)}}+\sum_{1\lesssim
L_1,\,L_2\lesssim N^4}\frac{N^{-\frac
32}}{N^{2s}L_{min}^{b-1/2}L_{med}^{b-1/2}N^{4(1-b)}}\\
&\lesssim &\sum_{1\lesssim L_1,\,L_2\lesssim N^4}\frac{1}
{N^{2s}L_{min}^{b-1/2}L_{med}^{\frac b2-\frac 14}N^{4(1-b)}N^{\frac
32(\frac 12+b)}}+\sum_{1\lesssim L_1,\,L_2\lesssim
N^4}\frac{N^{-\frac
32}}{N^{2s}L_{min}^{b-1/2}L_{med}^{b-1/2}N^{4(1-b)}}\\
&=:&I_{11}+I_{12}.
\end{eqnarray*}
Performing the $L$ summations, we see that $I_{11}\lesssim 1$ if
$2s+4(1-b)+\frac 32(\frac 12+b)>0$ and $I_{12}\lesssim 1$ if
$2s+4(1-b)+\frac 32>0$. This implies that $I_1\lesssim 1$ if
$s>-\frac 74$ and $\frac 12<b<\frac {4s+11}8$.

Now we consider the estimate of the second term $I_2$. The estimate
is a little simple compared to the estimate of $I_1$. We do not need
distinguish the cases $\Big(\frac{L_{med}}{N^3}\Big)^{1/2}\geq 1$
and $\Big(\frac{L_{med}}{N^3}\Big)^{1/2}<1$. We could get the
following estimate for $I_2$ in a unified way
\begin{eqnarray}\label{id11}
I_2&\leq &\sum_{1<N_3\ll N}\sum_{1\lesssim L_1,\,L_2\lesssim
N^4N_3}\frac{N_3N_3^s}{N^{2s}L_{min}^{b-1/2}L_{med}^b(N^4N_3)^{1-b}}
N^{-2}N^{1/2}L_{med}^{1/2}N_3^{-1/2}.
\end{eqnarray}
Performing the $N_3$ summation in (\ref{id11}) and noticing that if
$s-\frac 12+b<0$, we have
\begin{eqnarray*}
I_2&\lesssim & \sum_{1\lesssim L_1,\,L_2\lesssim
N^5}\frac{1}{N^{2s}L_{min}^{b-1/2}L_{med}^{b-1/2}N^{4(1-b)}N^{3/2}}\lesssim
1
\end{eqnarray*}
under condition that $2s+4(1-b)+\frac 32>0$. If $s-\frac 12+b\ge 0$,
we have
\begin{eqnarray*}
I_2&\lesssim & \sum_{1\lesssim L_1,\,L_2\lesssim
N^5}\frac{N^{s-1/2+b}}{N^{2s}L_{min}^{b-1/2}L_{med}^{b-1/2}N^{4(1-b)}N^{3/2}}\lesssim
1
\end{eqnarray*}
under condition that $2s+4(1-b)+\frac 32>s-\frac 12+b$. That means
that $I_2\lesssim 1$ if $s>-\frac 74$ and $\frac
12<b<\min\Big(\frac{4s+11}8,\,\frac{s+6}5\Big)$. Combining the
estimates for $I_1$ and $I_2$, we obtain the desired estimate
(\ref{id7}).

Now we deal with the case $N\sim N_2\sim N_3\gg N_1;\, H\sim
L_1\gtrsim L_2,\,L_3$. In this case we reduce by (\ref{estimate2})
to
\begin{eqnarray}\label{id12}
&&\sum_{N_1\ll N}\sum_{1\lesssim L_2,\,L_3\lesssim
N^4N_1}\frac{N^{1+s}L_{min}^{1/2}}{N^{s}<N_1>^sL_2^bL_3^{1-b}(N^4N_1)^b}
\min\Big(H,\,\frac N{N_1}L_{med}\Big)^{1/2}\lesssim 1.
\end{eqnarray}
Decompose the left-hand side of (\ref{id12}) into the following two
terms:
\begin{eqnarray*}\label{id13}
&&\sum_{N_1\leq 1}\sum_{1\lesssim L_2,\,L_3\lesssim
N^4N_1}\frac{N^{1+s}}{N^{s}L_2^bL_3^{1-b}(N^4N_1)^b}
L_{min}^{1/2}N_1^{1/2}+\\
&&\sum_{1<N_1\ll N}\sum_{1\lesssim L_2,\,L_3\lesssim
N^4N_1}\frac{N^{1+s}}{N_1^sN^{s}L_2^bL_3^{1-b}(N^4N_1)^b}
L_{min}^{1/2}N^{-2}N^{5/4}L_{med}^{1/4}=:J_1+J_2.
\end{eqnarray*}
In $J_1$, we assume $N_1\gtrsim N^{-4}$, otherwise the summation of
$L$ vanishes. Performing the summation of $L$, we get
\begin{eqnarray*}\label{id14}
J_1\lesssim \sum_{N^{-4}\lesssim N_1\leq 1}\frac{NN_1^{\frac
12-b}}{N^{4b}}\lesssim\frac{NN^{(-4)(\frac 12-b)}}{N^{4b}}\lesssim
1.
\end{eqnarray*}
If we take $\frac 12<b<\frac 34$ in $J_2$, then performing the
summation of $L$ implies that
\begin{eqnarray*}\label{id15}
J_2\lesssim \sum_{1\le N_1\ll
N}\frac{N^{1/4}N_1^{-s-b}}{N^{4b}}\lesssim\frac{N^{1/4}N^{-s-b}}{N^{4b}}\lesssim
1,
\end{eqnarray*}
if $4b+s+b-\frac 14>0$. The condition $4b+s+b-\frac 14>0$ is always
true if $s>-\frac 74$ and $\frac 12<b\leq 1$. So, $J_2\lesssim 1$ if
$s>-\frac 74$ and $\frac 12<b<\frac 34$. Combining the estimates for
$J_1$ and $J_2$, we get the needed estimate (\ref{id12}).

To finish the proof of (\ref{id3}) it remains to deal with the cases
where (\ref{estimate3}) holds. This reduces to
\begin{eqnarray}\label{id16}
&&\sum_{N_{max}\sim N_{med}\sim N}\sum_{L_{max}\sim N^4N_{min}}\nonumber\\
&&\quad\quad\frac{N_3<N_3>^{s}}{<N_1>^s<N_2>^sL_1^{b}
L_2^{b}L_3^{1-b}}
L_{min}^{1/2}N^{-2}\min{(H,\,L_{med})}^{1/2}\lesssim 1.
\end{eqnarray}
To estimate (\ref{id16}), by symmetry we need to consider two cases:
$N_1\sim N_2\sim N,\,N_3=N_{min}$ and $N_1\sim N_3\sim
N,\,N_2=N_{min}$.

(i) When $N_1\sim N_2\sim N,\,N_3=N_{min}$, the estimate
(\ref{id16}) further reduces to
\begin{eqnarray*}
&&\sum_{\substack{N_1\sim N_2\sim N\\N_3\ll N}}\sum_{L_{max}\sim
N^4N_3}\frac{N_3<N_3>^s}{N^{2s}L_{min}^bL_{med}^b(N^4N_3)^{1-b}}
L_{min}^{1/2}N^{-2}L_{med}^{1/2}\lesssim 1,
\end{eqnarray*}
then performing the $L$ summations, we reduce to
\begin{eqnarray*}
\sum_{N_3\ll
N}\frac{N_3<N_3>^s}{N^{2+2s}N^{4(1-b)}N_3^{1-b}}\lesssim 1,
\end{eqnarray*}
which is true if $2+2s+4(1-b)>0$. So, (\ref{id16}) is true if $s>-2$
and $\frac 12<b<\frac{s+3}2$.

(ii) When $N_1\sim N_3\sim N,\,N_2=N_{min}$, the estimate
(\ref{id16}) can be reduced to
\begin{eqnarray}\label{id17}
&&\sum_{\substack{N_1\sim N_3\sim N\\N_2\ll N}}\sum_{L_{max}\sim
N^4N_2}\frac{N^{1+s}L_{min}^{1/2}N^{-2}}{N^s<N_2>^sL_{min}^bL_{med}^bL_{max}^{1-b}}
\min(H,\,L_{med})^{1/2}\lesssim 1.
\end{eqnarray}
Before performing the $L$ summations, as before we need pay a little
more attention to the summation of $N_2$. Decompose the left-hand
side of (\ref{id17}) into the following two terms:
\begin{eqnarray*}
&&\quad\quad\sum_{N_2\leq 1}\sum_{L_{max}\sim
N^4N_2}\frac{N}{L_{min}^{b-1/2}L_{med}^b
L_{max}^{1-b}}N^{-2}L_{med}^{1/4}(N^4N_2)^{1/4}\\
&&+\sum_{1\leq N_2\leq N}\sum_{L_{max}\sim
N^4N_2}\frac{N}{N_2^sL_{min}^{b-1/2}L_{med}^bL_{max}^{(1-b)}}N^{-2}L_{med}^{1/2}\\
&&=:J_3+J_4.
\end{eqnarray*}
It is easily seen that $J_3\lesssim 1$ for any $\frac 12<b\le 1$.
For $J_4$, if $s+1-b\ge 0$, we always have $J_4\lesssim 1$ for any
$\frac 12<b\le 1$. If $s+1-b<0$, we have $J_4\lesssim 1$ under
condition that $4(1-b)+s+1+(1-b)>0$. So, (\ref{id17}) is true if
$s>-\frac 72$ and $\frac 12<b<\frac{s+6}5$. This completes the proof
of Proposition (\ref{nonl}).
\end{proof}

\section{A trilinear estimate and local well-posedness of the modified Kawahara equation}
\setcounter{equation}{0} \label{sec:5}

In this section we will prove a trilinear estimate in Bourgain
spaces, from which and the linear estimates presented in Section 2,
the local well-posedness of the initial-value problem for the
modified Kawahara equation Theorem\ref{major2} could be derived.

\begin{lemma}\label{trilinear} \, Let $s\ge -\frac 14$. For all $u_1,\,u_2,\,u_3$ on ${\Bbb
R}\times {\Bbb R}$ and $\frac 12<b\le 1$, we have
\begin{eqnarray}
\|\partial_x(u_1u_2u_3)\|_{X_{s,\,b-1}}\lesssim
\|u_1\|_{X_{s,\,b}}\|u_2\|_{X_{s,\,b}}\|u_3\|_{X_{s,\,b}}.
\end{eqnarray}
\end{lemma}
\noindent This is the first trilinear estimate in Bourgain spaces
associated to the class of Kawahara equations. It seems difficult to
obtain this kind of trilinear estimates by the method firstly
presented by Bourgain, Kenig-Ponce-Vega for KdV. We reduce the
trilinear estimate by the $TT^*$ identity Tao developed in
\cite{Tao2001} to a bilinear estimate, then prove the bilinear
estimate by the fundamental estimate on dyadic blocks in Lemma
\ref{fundamental estimate on dyadic blocks}.

\begin{proof}\, By duality and Plancherel it suffices to show that
\begin{eqnarray*}
\left\|\frac{(\xi_1+\xi_2+\xi_3)<\xi_4>^s}{<\tau_4-p(\xi_4)>^{1-b}
\prod_{j=1}^3<\xi_j>^s<\tau_j-p(\xi_j)>^b}\right\|_{[4,\,{\Bbb
R}\times{\Bbb R}]}\lesssim 1.
\end{eqnarray*}
We estimate $|\xi_1+\xi_2+\xi_3|$ by $<\xi_4>$. Applying the
fractional Leibnitz rule, we have
$$<\xi_4>^{s+1}\lesssim <\xi_4>^{1/2}\sum_{j=1}^3<\xi_j>^{s+1/2}$$
where we assume $s>-1/2$, and symmetry to reduce to
\begin{eqnarray*}
\left\|\frac{<\xi_1>^{-s}<\xi_3>^{-s}<\xi_2>^{1/2}<\xi_4>^{1/2}}{<\tau_4-p(\xi_4)>^{1-b}
\prod_{j=1}^3<\tau_j-p(\xi_j)>^{b}}\right\|_{[4,\,{\Bbb
R}\times{\Bbb R}]}\lesssim 1.
\end{eqnarray*}
We may replace $<\tau_2-p(\xi_2)>^{b}$ by $<\tau_2-p(\xi_2)>^{1-b}$
(this is true for any $b\ge \frac12$). By the $TT^*$ identity (see
Lemma 3.7 in \cite{Tao2001}, p847), the estimate is reduced to the
following bilinear estimate.
\end{proof}

\begin{lemma}(Bilinear estimate).\, Let $s\ge -\frac 14$. For all $u,\,v$ on ${\Bbb
R}\times {\Bbb R}$ and $0<\epsilon\ll 1$, we have
\begin{eqnarray}
\|uv\|_{L^2({\Bbb R}\times{\Bbb R})}\lesssim
\|u\|_{X_{-1/2,\,1/2-\epsilon}({\Bbb R}\times{\Bbb R})
}\|v\|_{X_{s,\,1/2+\epsilon}({\Bbb R}\times{\Bbb R})}.
\end{eqnarray}
\end{lemma}
\noindent This lemma can be proved in the same way as Proposition
\ref{nonl} by using the fundamental estimate on dyadic blocks in
Lemma \ref{fundamental estimate on dyadic blocks}. But we should
point out that there is some differences between this lemma and
Proposition \ref{nonl}. Lemma 5.2 is an asymmetric bilinear estimate
while Proposition \ref{nonl} is a symmetric bilinear estimate. This
leads to the lack of some symmetry in the proof of Lemma 5.2. On the
other hand, since there is no derivative in the left-hand side of
(5.2), the proof of Lemma 5.2 is rather simpler than that of
Proposition \ref{nonl}.

\begin{proof} By Plancherel it suffices to prove that
\begin{eqnarray}\label{norm5}
\left\|\frac{<\xi_1>^{-s}<\xi_2>^{1/2}}{<\tau_1-p(\xi_1)>^{1/2+\epsilon}
<\tau_2-p(\xi_2)>^{1/2-\epsilon}}\right\|_{[3,\,{\Bbb R}\times{\Bbb
R}]}\lesssim 1.
\end{eqnarray}
By dyadic decomposition and orthogonality as in the proof of
Proposition \ref{nonl}, we reduce the multiplier norm estimate
(\ref{norm5}) to showing that
\begin{eqnarray}\label{id20}
\sum_{N_{max}\sim N_{med}\sim N}\sum_{L_1,\,L_2,\,L_3\gtrsim
1}&&\frac{<N_1>^{-s}<N_2>^{1/2}}{L_1^{1/2+\epsilon}L_2^{1/2-\epsilon}}\nonumber\\
&&\left\|X_{N_1,\,N_2,\,N_3;\,L_{max};\,L_1,\,L_2,\,L_3}\right\|_{[3;\,{\Bbb
R}\times{\Bbb R}]}\lesssim 1
\end{eqnarray}
and
\begin{eqnarray}\label{id21}
\sum_{N_{max}\sim N_{med}\sim N}\sum_{L_{max}\sim L_{med}}\sum_{H\ll
L_{max}}&&\frac{<N_1>^{-s}<N_2>^{1/2}}{L_1^{1/2+\epsilon}L_2^{1/2-\epsilon}}
\nonumber\\
&&\left\|X_{N_1,\,N_2,\,N_3;\,H;\,L_1,\,L_2,\,L_3}\right\|_{[3;\,{\Bbb
R}\times{\Bbb R}]}\lesssim 1
\end{eqnarray}
for all $N\gtrsim 1$.

Fix $N\gtrsim 1$. We first prove (\ref{id21}). By (\ref{estimate3})
we reduce to
\begin{eqnarray}\label{id22}
&&\sum_{N_{max}\sim N_{med}\sim N}\sum_{L_{max}\sim L_{med}\gtrsim
N^4N_{min}}\frac{<N_1>^{-s}<N_2>^{1/2}}{L_1^{1/2+\epsilon}L_2^{1/2-\epsilon}}
L_{min}^{1/2}N_{min}^{1/2}\lesssim 1.
\end{eqnarray}
We consider two cases: $s\geq 0$ and $s<0$.

(i) In the first case $s\geq 0$, the estimate (\ref{id22}) can be
further reduced to
\begin{eqnarray*}
&&\sum_{N_{max}\sim N_{med}\sim N}\sum_{L_{max}\sim L_{med}\gtrsim
N^4N_{min}}\frac{N^{1/2}<N_{min}>^{-s}}{L_{min}^{1/2+\epsilon}L_{med}^{1/2-\epsilon}}
L_{min}^{1/2}N_{min}^{1/2}\lesssim 1,
\end{eqnarray*}
then performing the $L$ summations, we reduce to
\begin{eqnarray*}
\sum_{N_{max}\sim N_{med}\sim
N}\frac{N^{1/2}N_{min}^{1/2}<N_{min}>^{-s}}{(N^{4}N_{min})^{1/2-\epsilon}}\lesssim
1,
\end{eqnarray*}
which is always true for $s\geq 0$.

(ii) In the second case $s<0$, the estimate (\ref{id22}) can be
reduced to
\begin{eqnarray*}
&&\sum_{N_{max}\sim N_{med}\sim N}\sum_{L_{max}\sim L_{med}\gtrsim
N^4N_{min}}\frac{N^{1/2-s}}{L_{min}^{1/2+\epsilon}(N^4N_{min})^{1/2-\epsilon}}
L_{min}^{1/2}N_{min}^{1/2}\lesssim 1.
\end{eqnarray*}
Performing the $L$ summations, we reduce to
\begin{eqnarray*}
&&\sum_{N_{max}\sim N_{med}\sim
N}\frac{N^{1/2-s}N_{min}^{\epsilon}}{N^{2-4\epsilon}}\lesssim 1,
\end{eqnarray*}
which is true if $s> -\frac 32$. So, (\ref{id22}) is true if
$s>-\frac 32$.

Now we show the low modulation case (\ref{id20}). In this case we
may assume $L_{max}\sim N_{max}^4N_{min}$. We first deal with the
contribution where (\ref{estimate1}) holds. In this case we have
$N_1,\,N_2,\,N_3\sim N\gtrsim 1$, so we reduce to
\begin{eqnarray}\label{id23}
\sum_{L_{max}\sim
N^5}\frac{N^{-s}N^{1/2}}{L_{min}^{1/2+\epsilon}L_{med}^{1/2-\epsilon}}
L_{min}^{1/2}N^{-2}L_{med}^{1/2}\lesssim 1.
\end{eqnarray}
Performing the $L$ summations, we reduce to
$$\frac{N^{5\epsilon}}{N^{3/2+s}}\lesssim 1,$$
which is true if $s>-\frac 32$.

Now we deal with the cases where (\ref{estimate2}) holds. Since the
lack of symmetry, we need to consider three cases
$$\begin{array}{ll}
N\sim N_1\sim N_2\gg N_3;\,& H\sim L_3\gtrsim L_1,\,L_2\\
N\sim N_2\sim N_3\gg N_1;\,& H\sim L_1\gtrsim L_2,\,L_3\\
N\sim N_1\sim N_3\gg N_2;\,& H\sim L_2\gtrsim L_1,\,L_3.
\end{array}$$

In the first case we reduce by (\ref{estimate2}) to
\begin{eqnarray}\label{id24}
&&\sum_{N_3\ll N}\sum_{1\lesssim L_1,\,L_2\lesssim
N^4N_3}\frac{N^{1/2-s}}{L_{min}^{1/2+\epsilon}L_{med}^{1/2-\epsilon}}
L_{min}^{1/2}N^{-2}\min\Big(N^4N_3,\,\frac{N}{N_3}L_{med}\Big)^{1/2}\lesssim
1.
\end{eqnarray}
Performing the $N_3$ summation we reduce to
\begin{eqnarray*}\label{id25}
&&\sum_{1\lesssim L_1,\,L_2\lesssim
N^4N_3}\frac{N^{1/2-2-s+5/4}}{L_{min}^{1/2+\epsilon}L_{med}^{1/2-\epsilon}}
L_{min}^{1/2}L_{med}^{1/4}\lesssim 1
\end{eqnarray*}
which is true if $s\geq -\frac 14$.

Now we deal with the second case $N\sim N_2\sim N_3\gg N_1;\, H\sim
L_1\gtrsim L_2,\,L_3$. In this case we make use of the first half of
(\ref{estimate2}) and reduce to
\begin{eqnarray}\label{id26}
&&\sum_{N_1\ll N}\sum_{1\lesssim L_2,\,L_3\lesssim
N^4N_1}\frac{N^{1/2}}{<N_1>^sL_2^{1/2-\epsilon}(N^4N_1)^{1/2+\epsilon}}
L_{min}^{1/2}N_1^{1/2}\lesssim 1
\end{eqnarray}
which is true if $s>-1$.

In the third case $N\sim N_1\sim N_3\gg N_2;\, H\sim L_2\gtrsim
L_1,\,L_3$, we similarly reduce by using the first half of
(\ref{estimate2}) to
\begin{eqnarray}\label{id26}
&&\sum_{N_2\ll N}\sum_{1\lesssim L_1,\,L_3\lesssim
N^4N_2}\frac{<N_2>^{1/2}}{N^sL_1^{1/2+\epsilon}(N^4N_2)^{1/2-\epsilon}}
L_{min}^{1/2}N_2^{1/2}\lesssim 1
\end{eqnarray}
which is true if $s>-\frac 32$.

To finish the proof of (\ref{id20}) it remains to deal with the
cases where (\ref{estimate3}) holds. This reduces to
\begin{eqnarray}\label{id27}
\sum_{N_{max}\sim N_{med}\sim N}\sum_{L_{max}\sim N^4N_{min}}
\frac{<N_2>^{1/2}}{<N_1>^sL_1^{1/2+\epsilon} L_2^{1/2-\epsilon}}
L_{min}^{1/2}N^{-2}L_{med}^{1/2}\lesssim 1
\end{eqnarray}
which is true if $s>-\frac 32$. This finishes the proof of Lemma
5.2.
\end{proof}

\vspace{.3in} \noindent{\bf Acknowledgements}: The authors would
like to express their gratitude to the anonymous referee and the
associated editor for their invaluable comments and suggestions
which helped improve the paper greatly. We also thank Justin Holmer
and Juan-Ming Yuan for discussions and for pointing to us some of
the references. W.\,Chen, J.\,Li and C.\,Miao are supported by the
NNSF of China (No.10725102, No.10771130, No.10626008).


\begin{thebibliography}{99}
\bibitem{Bou} J. Bourgain, Fourier retsriction phenomena for certain
lattice subsets and applications to nonlinear evolution equations,
{\sl Geom. Funct. Anal.} {\bf 3} (1993), 107-156, 209-262.

\bibitem{CDT} S.-B. Cui, D.-G. Deng and S.-P. Tao, Global
existence of solutions for the Cauchy problem of the Kawahara
equation with $L^{2}$ initial data, {\sl Acta Math. Sin.} {\bf
22}(2006), 1457-1466.

\bibitem{Ka} T. Kawahara, Oscillatory solitary waves in dispersive
media, {\it J. Phys. Soc. Japan} {\bf 33} (1972), 260-264.

\bibitem{KPV91} C.E. Kenig, G. Ponce and L. Vega, Oscillatory
integrals and regularity of dispersive equations, {\sl Indiana Univ.
Math. J.} {\bf 40} (1991), 33-69.

\bibitem{KPV93a} C.E. Kenig, G. Ponce and L. Vega, Well-posedness and
scattering results for the genralized KdV equation via the
contraction principle, {\sl Comm. Pure Appl. Math.} {\bf 46} (1993),
527-620.

\bibitem{KPV93b} C.E. Kenig, G. Ponce and L. Vega, The Cauchy
problem for the Korteweg-De Vries equation in Sobolev spaces of
negative indices, {\sl Duke Math. J.} {\bf 71} (1993), 1-21.

\bibitem{KPV94a} C.E. Kenig, G. Ponce and L. Vega, On the hierarchy
of the generalized KdV equations, Proc. Lyon Workshop on {\it
Singular Limits of Dispersive Waves} (Lyon, 1991), 347-356, {\it
NATO Adv. Sci. Inst. Ser. B Phys.}, {\bf 320}, Plenum, New York,
1994.

\bibitem{KPV94b} C.E. Kenig, G. Ponce and L. Vega, Higher-order
nonlinear dispersive equations, {\sl Proc. Amer. Math. Soc.} {\bf
122} (1994), 157-166.

\bibitem{KPV96} C.E. Kenig, G. Ponce and L. Vega, A bilinear
estimate with applications to the KdV equation, {\sl J. Amer. Math.
Soc.} {\bf 9} (1996), 573-603.

\bibitem{KO}S. Kichenassamy and P.J. Olver,
Existence and nonexistence of solitary wave solutions to high-order
model evolution equations, {\sl SIAM J. Math. Anal.} {\bf 23}(1992),
1141-1166.

\bibitem{Po}G. Ponce, Lax pairs and higher order models for water
waves, {\sl J. Differential Equations} {\bf 102} (1993), 360-381.

\bibitem{St}E.M. Stein, {\it Harmonic Analysis: Real-Variable
Methods, Orthogonality and Oscillatory Integrals}, Princeton
Unviersity Press, Princeton, NJ, 1993.

\bibitem{Taocui}S.-P.Tao and S.-B. Cui, Local and global existence of
solutions to initial value problems of modified nonlinear Kawahara
equations, {\sl Acta Math. Sin.} {\bf 21} (2005), 1035-1044.

\bibitem{Tao2001}T. Tao, Multilinear weighted convolution of
$L^2$ functions, and applications to nonlinear dispersive equations,
{\sl Amer. J. Math.} {\bf 123} (2001), 839-908.

\bibitem{WCD} H. Wang, S.-B. Cui and D.-G. Deng, Global
existence of solutions for the Kawahara equation in Sobolev spaces
of negative indices, {\sl Acta Math. Sin.} {\bf 23}(2007),
1435-1446.

\end{thebibliography}
\end{document}